\documentclass[10pt,leqno]{amsart}
\usepackage{amsmath}
\usepackage{amsfonts}
\usepackage{amssymb}
\usepackage{graphicx}
\usepackage{color}
\usepackage{hyperref}
\usepackage{xcolor}
\newtheorem{theorem}{Theorem}[section]
\newtheorem*{theorem*}{Theorem}
\newtheorem{lemma}[theorem]{Lemma}
\newtheorem{corollary}[theorem]{Corollary}
\newtheorem{proposition}[theorem]{Proposition}

\theoremstyle{definition}

\newtheorem{example}[theorem]{Example}

\newcommand{\R}{\mathbb{R}}

\def \b {\beta}

\def\Ric{\text{Ric}}

\def\a{\alpha}
\def\l{\lambda}

\def\R{\mathbb{R}}
\def\Z{\mathbb{Z}}

\def\vp{\varphi}

\def\id{\operatorname{id}}
\def\Ric{\operatorname{Ric}}

\def\tr{\operatorname{tr}}

\def\spn{\operatorname{span}}

\numberwithin{equation}{section}
\makeatletter
\newcommand*\owedge{\mathpalette\@owedge\relax}
\newcommand*\@owedge[1]{%
  \mathbin{%
    \ooalign{%
      $#1\m@th\bigcirc$\cr
      \hidewidth$#1\m@th\wedge$\hidewidth\cr
    }%
  }%
}
\makeatother

\DeclareMathOperator{\spann}{span}

\begin{document}

\title[The curvature operator of the second kind in 3D]{The curvature operator of the second kind in dimension three}

\author[Fluck]{Harry Fluck}
\address{Department of Mathematics, Cornell University \\ Ithaca, NY 14853, USA}
\email{hpf5@cornell.edu}

\author[Li]{Xiaolong Li}\thanks{The second author's research is partially supported by NSF LEAPS-MPS \#2316659, Simons Collaboration Grant \#962228, and a start-up grant at Wichita State University}
\address{Department of Mathematics, Statistics and Physics, Wichita State University, Wichita, KS, 67260, USA}
\email{xiaolong.li@wichita.edu}



\begin{abstract}
This article aims to understand the behavior of the curvature operator of the second kind under the Ricci flow in dimension three.
First, we express the eigenvalues of the curvature operator of the second kind explicitly in terms of that of the curvature operator (of the first kind). Second, we prove that $\a$-positive/$\a$-nonnegative curvature operator of the second kind is preserved by the Ricci flow in dimension three for all $\a \in [1,5]$.
\end{abstract}
\maketitle

\section{Introduction}

The Riemann curvature tensor $R_{ijkl}$ on an $n$-dimensional Riemannian manifold $(M^n,g)$ naturally induces two self-adjoint curvature operators: $\hat{R}$ acts on the space of two-forms $\wedge^2(T_pM)$ via 
\begin{equation*}
    \hat{R}(\omega)_{ij} =\frac{1}{2}\sum_{k,l=1}^n R_{ijkl}\omega_{kl},
\end{equation*}
and $\overline{R}$ acts on the space of symmetric two-tensors $S^2(T_pM)$ via 
\begin{equation*}
\overline{R}(\vp)_{ij} =\sum_{k,l=1}^n R_{iklj}\vp_{kl}.
\end{equation*}
In the literature, $\hat{R}$ is known as the curvature operator and there are many remarkable results under various positivity conditions on $\hat{R}$; see Meyer \cite{Meyer71}, Gallot and Meyer \cite{GM75}, Tachibana \cite{Tachibana74}, Hamilton \cite{Hamilton82, Hamilton86}, B\"ohm and Wilking \cite{BW08}, Brendle and Schoen \cite{BS08, BS09}, Andrews and Nguyen \cite{AN09}, Ni and Wilking \cite{NW10}, Petersen and Wink \cite{PW21, PW21Crelle, PW22}, etc.  
In particular, the celebrated differentiable sphere theorem states that closed manifolds with two-positive curvature operator are diffeomorphic to spherical space forms. This is proved using the Ricci flow by Hamilton \cite{Hamilton82} for $n=3$, Hamilton \cite{Hamilton86} and Chen \cite{Chen91} for $n=4$, and B\"ohm and Wilking \cite{BW08} for $n\geq 5$. Here, $\hat{R}$ is two-positive if the sum of the smallest two eigenvalues of $\hat{R}$ is positive, and $(M^n,g)$ is said to have two-positive curvature operator if $\hat{R}_p$ is two-positive at every $p\in M$. 
The corresponding classification of closed manifolds with two-nonnegative curvature operator is due to Hamilton \cite{Hamilton86} for $n=3$, Hamilton \cite{Hamilton86} and Chen \cite{Chen91} for $n=4$, and Ni and Wu \cite{NW07} for $n\geq 5$. We refer the reader to the wonderful surveys \cite{Wilking07}, \cite{BS09Survey}, and \cite{Ni14} for more information. 

\textit{The curvature operator of the second kind}, denoted by $\mathring{R}$ throughout this article, refers to the restriction of $\overline{R}$ to $S^2_0(T_pM)$, the space of traceless symmetric two-tensors. Nishikawa \cite{Nishikawa86} interpreted $\mathring{R}$ as the symmetric bilinear form 
$$\mathring{R}: S^2_0(T_pM) \times S^2_0(T_pM) \to \R$$ 
obtained by restricting $\overline{R}$ to $S^2_0(T_pM)$. He called $\mathring{R}$ the curvature operator of the second kind, to distinguish it from the curvature operator $\hat{R}$, which he called the curvature operator of the first kind. It was pointed out in \cite{NPW22} that the curvature operator of the second kind can also be interpreted as the self-adjoint operator 
$$\mathring{R}=\pi \circ \overline{R} :S^2_0(T_pM) \to S^2_0(T_pM),$$ 
where $\pi:S^2(T_pM) \to S^2_0(T_pM)$ is the projection map. The algebraic reason to restrict to $S^2_0(T_pM)$, as pointed out by Bourguignon and Karcher \cite{BK78}, is that $S^2(T_pM)$ is not irreducible under the action of the orthogonal group $O(T_pM)$. Indeed, it splits into $O(T_pM)$-irreducible subspaces as 
$$S^2(T_pM) =S^2_0(T_pM) \oplus \R g.$$ 
Geometrically, $\mathring{R}=\pi \circ \overline{R}=\id_{S^2_0}$ on the standard sphere $\mathbb{S}^n$ with constant sectional curvature $1$, while the operator $\overline{R}$ is not even positive: the eigenvalues of $\overline{R}$ on $\mathbb{S}^n$ are given by $-(n-1)$ with multiplicity one and $1$ with multiplicity $\frac{(n-1)(n+2)}{2}$ (see \cite{BK78}).

Recently, the curvature operator of the second kind $\mathring{R}$ received attention due to the resolution of Nishikawa's 1986 conjecture, which states that a closed Riemannian manifold with positive (respectively, nonnegative) curvature operator of the second kind is diffeomorphic to a spherical space form (respectively, Riemannian locally symmetric space). 
The positive part was resolved by Cao, Gursky, and Tran \cite{CGT21}. Their key observation is that two-positivity of $\mathring{R}$ implies the strict PIC1 condition introduced by Brendle \cite{Brendle08}, i.e., $M\times \R$ has positive isotropic curvature. The positive part of Nishikawa's conjecture then follows from Brendle's convergence result \cite{Brendle08} stating that the normalized Ricci flow evolves an initial metric that is strictly PIC1 into a limit metric with constant positive sectional curvature. 
Shortly after, the second named author \cite{Li21} noticed that strictly PIC1 is implied by three-positivity of $\mathring{R}$, thus getting an improvement of the result of Cao, Gursky, and Tran.
He also settled the nonnegative part of Nishikawa's conjecture by reducing it to the locally irreducible case and appealing to the classification of closed locally irreducible manifolds with weakly PIC1 in \cite{Brendle10book}.  

After that, further investigations toward understanding $\mathring{R}$ have been carried out in \cite{NPW22}, \cite{NPWW22},  \cite{Li22JGA, Li22PAMS, Li22Kahler, Li22product}, and \cite{ZZ22}. The second named author \cite{Li22JGA} proved that closed manifolds with $4\frac{1}{2}$-positive $\mathring{R}$ have positive isotropic curvature and positive Ricci curvature, thus being homeomorphic to spherical space forms because of the work of Micallef and Moore \cite{MM88}. This improves a result of Cao, Gursky, and Tran \cite[Theorem 1.6]{CGT21} assuming $4$-positivity of $\mathring{R}$. 
Using the Bochner technique, Nienhaus, Petersen, and Wink \cite{NPW22} proved vanishing results on the Betti numbers $b_p$ under $C(p,n)$-positivity of $\mathring{R}$, where $C(p,n)$ is an explicit constant. In particular, it follows that a closed Riemannian $n$-manifold with $\frac{n+2}{2}$-nonnegative $\mathring{R}$ is either flat or a rational homology sphere.
Together with Wylie \cite{NPWW22}, they observed that a Riemannian $n$-manifold with $n$-nonnegative or $n$-nonpositive $\mathring{R}$ has restricted homology $SO(n)$ unless it is flat. Subsequently, the second named author obtained a sharper result in his investigation of $\mathring{R}$ on product manifolds \cite{Li22product}.  
In addition, $\mathring{R}$ was investigated on K\"ahler manifold in \cite{Li21}, \cite{Li22PAMS}, \cite{NPWW22}, and  \cite{Li22Kahler}. It is proved in \cite{Li22Kahler} that a K\"ahler manifold of complex dimension $m$ with $\a$-nonnegative $\mathring{R}$ must be flat if $\a < \frac{3}{2}(m^2-1)$. This is sharp as $\mathbb{CP}^m$ with the Fubini-Study metric has $\frac{3}{2}(m^2-1)$-nonnegative $\mathring{R}$. 
Another result in \cite{Li22Kahler} asserts that a closed K\"ahler manifold of complex dimension $m$ with $\left(\frac{3m^3-m+2}{2m} \right)$-positive $\mathring{R}$ has positive orthogonal bisectional curvature, thus being biholomorphic to $\mathbb{CP}^m$. In another direction, Zhao and Zhu \cite{ZZ22} proved that a steady gradient Ricci soliton of dimension $n\geq 4$ must be isometric to the Bryant soliton up to scaling if it is asymptotically cylindrical and satisfies $\overline{R}>-\frac{S}{2} \id_{S^2(T_pM)}$, where $S$ denotes the scalar curvature. This improves a result of Brendle \cite{Brendle14} assuming the stronger assumption of positive sectional curvature.

In this article, we investigate the curvature operator of the second kind in dimension three. 
Our first result finds explicit expressions for the eigenvalues of $\mathring{R}$ in terms of that of $\hat{R}$ in this dimension.  
\begin{theorem}\label{thm eigenvalue 3D}
Let $R\in S^2_B(\wedge^2 \R^3)$ be an algebraic curvature operator on $\R^3$ and denote by $a\leq b \leq c$ the eigenvalues of the curvature operator $\hat{R}$. Then the eigenvalues of the curvature operator of the second kind $\mathring{R}$ are given by 
\begin{equation*}
    \l_{-} \leq a\leq b \leq c \leq  \l_{+},
\end{equation*}
where
\begin{equation}\label{def lambda pm}
    \l_{\pm}:= \frac{a+b+c}{3} \pm \frac{\sqrt{2}}{3}\sqrt{3(a^2+b^2+c^2)-(a+b+c)^2}.
\end{equation}
\end{theorem}

It is well-known that both kinds of curvature operators are stronger than sectional curvature, in the sense that $\hat{R} \geq \kappa$ or $\mathring{R} \geq \kappa$ for some $\kappa \in \R$ implies that the sectional curvatures are bounded from below by $\kappa$ (see for example \cite{Nishikawa86}). 
In dimension three, $\hat{R} \geq 0$ is equivalent to nonnegative sectional curvature, and two-nonnegativity of $\hat{R}$ is equivalent to nonnegative Ricci curvature. In all dimensions, two-nonnegativity of $\mathring{R}$ implies nonnegative sectional curvature (see \cite{Li21}) and $(n+\frac{n-2}{n})$-nonnegativity of $\mathring{R}$ implies nonnegative Ricci curvature (see \cite{Li22JGA}). Here $\mathring{R}$ is said to be $\a$-nonnegative for some $\a \in [1,\dim(S^2_0(T_pM))]$ if the eigenvalues  $\l_1 \leq \l_2 \leq \cdots \leq \l_{\frac{(n-1)(n+2)}{2}}$ of $\mathring{R}$ satisfies
\begin{equation*}
    \l_1 +\cdots + \l_{\lfloor \a \rfloor} +(\a -\lfloor \a \rfloor)\l_{\lfloor \a \rfloor+1} \geq 0, 
\end{equation*}
where $\lfloor x \rfloor$ denotes the floor function defined by 
$$\lfloor x \rfloor := \max \{k \in \Z: k \leq x\}.$$
When $\a=k$ is an integer, this agrees with the usual definition meaning that the sum of the smallest $k$ eigenvalues of $\mathring{R}$ is nonnegative. The $\a$-positivity of $\mathring{R}$ is defined similarly. Moreover, we say $\mathring{R}$ is $\a$-nonpositive (resp, $\a$-negative) if $-\mathring{R}$ is $\a$-nonnegative (resp, $\a$-positive). We always omit $\a$ when $\a=1$. 

Clearly, $\a$-positive $\mathring{R}$ implies $\b$-positive $\mathring{R}$ if $\a \leq \b$. Therefore, $\a$-positive $\mathring{R}$ for $\a \in [1, \frac{(n-1)(n+2)}{2}]$ provides a family of curvature conditions interpolating between positive curvature operator of the second kind (corresponding to $\a=1$) and positive scalar curvature (corresponding to $\a=\frac{(n-1)(n+2)}{2})$. In seeking optimal $\a$-positivity of $\mathring{R}$ that characterizes the spherical space form, the second named author \cite{Li22JGA} conjectured that a closed Riemannian manifold with $\left(n+\frac{n-2}{n}\right)$-positive curvature operator of the second kind is diffeomorphic to a spherical space form, after verifying it in dimensions three and four. The number $n+\frac{n-2}{n}$ comes from the standard cylinder $\mathbb{S}^{n-1}\times \mathbb{S}^1$, which has $\alpha$-positive $\mathring{R}$ if $\a > n+\frac{n-2}{n}$.

With the aid of Theorem \ref{thm eigenvalue 3D}, we get some new and optimal results in dimension three. We summarize (noting that parts (1), (3), and (5) have been proved by the second named author in \cite{Li21, Li22JGA}) that
\begin{proposition}\label{prop 3D curvature}
Let $R\in S^2_B(\wedge^2 \R^3)$ be an algebraic curvature operator and denote by $\hat{R}$ and 
$\mathring{R}$ the curvature operator of the first and the second kind of $R$ respectively. 
Then the following statements hold. 
\begin{enumerate}
    \item $\mathring{R}$ is two-nonnegative $\implies$ $R$ has nonnegative sectional curvature.
    \item $R$ has nonnegative sectional curvature $\implies$ $\mathring{R}$ is $3\frac{1}{3}$-nonnegative.
    \item $\mathring{R}$ is $3\frac{1}{3}$-nonnegative $\implies$ $R$ has nonnegative Ricci curvature
    \item $R$ has nonnegative Ricci curvature $\implies$ $\mathring{R}$ is four-nonnegative 
   \item  $R$ has nonnegative scalar curvature $\iff$ $\mathring{R}$ is five-nonnegative. 
\end{enumerate}
Moreover, the same results hold if ``nonnegative" is replaced by ``positive", ``negative", or ``non-positive''. 
\end{proposition}

Our second result states that $\a$-positivity (and $\a$-nonnegativity) of $\mathring{R}$ is preserved by three-dimensional compact Ricci flows for every $\a \in [1,5]$. Recall that a Riemannian manifold $(M^n,g)$ is said to have $\a$-positive curvature operator of the second kind if $\mathring{R}_p$ is $\a$-positive for each $p \in M$.
\begin{theorem}\label{thm RF preserves}
Let $(M^3,g(t))$, $t\in [0,T)$, be a three-dimensional compact Ricci flow. Let $\a \in [1,5]$. If $g(0)$ has $\a$-positive (respectively, $\a$-nonnegative) curvature operator of the second kind, then $g(t)$ has $\a$-positive (respectively, $\a$-nonnegative) curvature operator of the second kind for $t\in (0,T)$.
\end{theorem}

It remains an interesting question whether the Ricci flow preserves $\a$-positive/$\a$-nonnegative curvature operator of the second kind for some $\a \in [1,N]$ in higher dimensions.

Since $3\frac{1}{3}$-nonnegativity of $\mathring{R}$ implies nonnegative Ricci curvature, one gets classification results for complete three-manifolds with $3\frac{1}{3}$-nonnegative $\mathring{R}$ via the classification results under nonnegative Ricci curvature in \cite{Hamilton82, Hamilton86} in the compact case and \cite{Liu13} in the complete noncompact case. Here, we get a refined version. 
\begin{theorem}\label{thm 3-manifold classification}
Let $(M^3,g)$ be an oriented complete three-dimensional Riemannian manifold with $\a$-nonnegative curvature operator of the second kind.
\begin{enumerate}
    \item If $\a \in [1,3\frac{1}{3})$, then $M$ is either flat or diffeomorphic to a spherical space form. 
    \item If $\a =3\frac{1}{3}$ and $M$ is closed, then $M$ is diffeomorphic to a quotient of one of the spaces $\mathbb{S}^3$ or $\mathbb{S}^2 \times \R$ or $\mathbb{R}^3$ by a group of fixed point free isometries in the standard metrics.
    \item If $\a =3\frac{1}{3}$ and $M$ is noncompact, then either $M$ is diffeomorphic to $\R^3$ or the universal cover of $M$ is isometric to a Riemann product $N^2 \times \R$ where $N^2$ is a complete two-manifold with nonnegative sectional curvature.
    \item If $\a \in (3\frac{1}{3},5]$ and $M$ is closed, then $M$ is either flat or diffeomorphic to a connected sum of spherical space forms and copies of $\mathbb{S}^2 \times \mathbb{S}^1$.
\end{enumerate}
\end{theorem}

Note that parts (2) and (3) have been proved by the second name author in \cite{Li22JGA}. Part (4) is a direct consequence of the celebrated work of Perelman \cite{Perelman1, Perelman2, Perelman3}. Here we prove part (1) using Proposition \ref{prop 3.3} and the recent resolution of a conjecture of Hamilton in dimension three by \cite{CZ00}, \cite{Lott19}, \cite{DSS22}, and \cite{LT22}. 

Before concluding this section, we would like to mention that the action of the Riemann curvature tensor on symmetric two-tensors indeed has a long history. It perhaps first appeared for K\"ahler manifolds in the study of the deformation of complex analytic structures by Calabi and Vesentini \cite{CV60}. They introduced the self-adjoint operator $\xi_{\a \b} \to R^{\rho}_{\ \a\b}{}^{\sigma} \xi_{\rho \sigma}$ from $S^2(T^{1,0}_p M)$ to itself and computed the eigenvalues of this operator on Hermitian symmetric spaces of classical type, with the exceptional ones handled shortly after by Borel \cite{Borel60}. In the Riemannian setting, the action arises naturally in the context of deformations of Einstein structure in Berger and Ebin \cite{BE69} (see also \cite{Koiso79a, Koiso79b} and \cite{Besse08}). 
In addition, it appears in the Bochner-Weitzenb\"ock formulas for symmetric two-tensors (see for example \cite{MRS20}), for differential forms in \cite{OT79}, and for Riemannian curvature tensors in \cite{Kashiwada93}.  
In another direction, curvature pinching estimates for $\overline{R}$ were studied by Bourguignon and Karcher \cite{BK78}, and they calculated the eigenvalues of $\overline{R}$ on the complex projective space with the Fubini-Study metric and the quaternionic projective space with its canonical metric. 
Nevertheless, the operators $\overline{R}$ and $\mathring{R}$ are significantly less investigated than $\hat{R}$ and it is our goal to gain a better understanding of them.

This paper is organized as follows. 
In Section 2, we fix some notations and conventions.
In Section 3, we provide a strategy to diagonalize the curvature operator of the second kind of an algebraic curvature operator with vanishing Weyl tensor.  
In Section 4, we prove Theorem \ref{thm eigenvalue 3D}, Proposition \ref{prop 3D curvature} and Theorem \ref{thm 3-manifold classification}.  
In Section 5, we prove Theorem \ref{thm RF preserves}.

\section{Notations and Conventions}

Throughout this article, $(V,g)$ is a real Euclidean vector space of dimension $n \geq 2$ and $\{e_i\}_{i=1}^n$ is an orthonormal basis of $V$ for the metric $g$. We always identify $V$ with its dual space $V^*$ via $g$. 

$S^2(V)$ and $\wedge^2V$ denote the space of symmetric two-tensors and two-forms on $V$, respectively. Note that $S^2(V)$ splits into $O(V)$-irreducible subspaces as
\begin{equation*}
        S^2(V)=S^2_0(V)\oplus \R g,
\end{equation*}
where $S^2_0(V)$ is the space of traceless symmetric two-tensors.

$S^2(\wedge^2 V)$, the space of symmetric two-tensors on $\wedge^2V$, has the orthogonal decomposition 
    \begin{equation*}
        S^2(\wedge^2 V) =S^2_B(\wedge^2 V) \oplus \wedge^4 V,
    \end{equation*}
where $S^2_B(\wedge^2 V)$ consists of all tensors $R\in S^2(\wedge^2V)$ that also satisfy the first Bianchi identity. The space $S^2_B(\wedge^2V)$ is called the space of algebraic curvature operators on $V$.

The tensor product is defined as $(e_i\otimes e_j)(e_k,e_l) =\delta_{ik}\delta_{jl}$, where $\delta_{pq}$ denotes the Kronecker delta defined as
$$\delta_{pq}=\begin{cases}
    1,  & \text{ if } p=q \\ 0, &\text{otherwise}.
\end{cases}$$
The symmetric product $\odot$ and wedge product $\wedge$ are defined as 
$$ u \odot v=u\otimes v +v \otimes u$$ and 
$$u \wedge v=u\otimes v - v \otimes u,$$ 
respectively. 

The inner product on $S^2(V)$ is given by $\langle A, B \rangle =\tr(A^T B)$ and the inner product on $\wedge^2 V$ is given by $ \langle A, B \rangle =\frac{1}{2}\tr(A^T B)$. If $\{e_i\}_{i=1}^n$ is an orthonormal basis of $V$, then $\{\frac{1}{\sqrt{2}}e_i \odot e_j\}_{1\leq i<j\leq n} \cup \{\frac{1}{2}e_i \odot e_i\}_{1\leq i\leq n}$ is an orthonormal basis of $S^2(V)$ and $\{e_i \wedge e_j\}_{1\leq i<j\leq n}$ is an orthonormal basis of $\wedge^2 V$. 

Given $A, B\in S^2(V)$, their Kulkarni-Nomizu product gives rise to $A \owedge B  \in  S^2_B(\wedge^2 V)$ via
\begin{equation*}
    (A \owedge B )_{ijkl} =A_{ik}B_{jl}+A_{jl}B_{ik} -A_{jk}B_{il}-A_{il}B_{jk}.
\end{equation*}
It is well known that $R\in  S^2_B(\wedge^2 V)$ can be written as $A \owedge \id$ for some $A\in S^2(V)$ if and only if the Weyl tensor of $R$ vanishes. 
In particular, any $R\in  S^2_B(\wedge^2 \R^3)$ can be written as $A\owedge \id$ for some $A\in S^2(V)$ .

\section{Diagonalization} 
The main result of this section is the following theorem, which expresses the eigenvalues of the curvature operator of the second kind of $A \owedge \id $ in terms of the eigenvalues of $A$.

\begin{theorem}\label{thm diagonilization}
Let $A\in S^2(V)$ and denote by $\mu_1, \cdots, \mu_k$ the distinct eigenvalues of $A$ with corresponding multiplicities $n_1, \cdots, n_k$.  
Then the eigenvalues of the curvature operator of the second kind of  $A \owedge \id$ are given by
\begin{enumerate}
    \item $\mu_i +\mu_j$ with multiplicity $n_in_j$, for $1\leq i < j \leq k$;
    \item  $2 \mu_i$  with multiplicity $n_i-1$,   for $1 \leq i \leq k$;
    \item  the $k-1$ nonzero solutions of the equation $\sum_{i=1}^k \frac{n_i \mu_i}{2\mu_i-\l}=\frac{n}{2}$ if $\sum_{i=1}^k \frac{n_i}{\mu_i} \neq 0$, and the $k-2$ nonzero solutions of $\sum_{i=1}^k \frac{n_i \mu_i}{2\mu_i-\l}=\frac{n}{2}$ together with $0$ if $\sum_{i=1}^k \frac{n_i}{\mu_i} =0$. 
\end{enumerate}
\end{theorem} 

To prove Theorem \ref{thm diagonilization}, we first observe that if $\l_i$ and $\l_j$ are eigenvalues of $A$, then $\l_i+\l_j$ is an eigenvalue of the curvature operator of the second kind of $A \owedge \id$. 
\begin{lemma}\label{lemma 3.1}
Let $A\in S^2(V)$ and $\{e_i\}_{i=1}^n$ be an orthonormal basis of $V$ such that $A(e_i)=\l_i e_i$ for $1\leq i \leq n$. Then we have
\begin{equation*}
    (A \owedge \id ) (e_p \odot e_q) =(\l_p+\l_q) (e_p \odot e_q)
\end{equation*}
for $1\leq p \neq  q\leq n$.
\end{lemma}

\begin{proof}
Note that
\begin{equation*}
    (e_p\odot e_q)(e_j,e_k)=\delta_{jp}\delta_{kq} + \delta_{kp}\delta_{jq}.
\end{equation*}
The conclusion follows from a straightforward calculation as follows:
\begin{eqnarray*}
 && \left((A \owedge \id ) (e_p \odot e_q)\right)_{il} \\
 &=&  (A \owedge \id )_{ijkl} (e_p \odot e_q)_{jk}  \\
&=&   (A_{ik}\delta_{jl}+A_{jl}\delta_{ik}-A_{il}\delta_{jk}-A_{jk}\delta_{il}) (\delta_{jp}\delta_{kq} +\delta_{jq}\delta_{kp}) \\
&=&  (\l_i \delta_{ik}\delta_{jl}+\l_j \delta_{jl}\delta_{ik}-\l_i \delta_{il}\delta_{jk}-\l_j\delta_{jk}\delta_{il})(\delta_{jp}\delta_{kq}+\delta_{jq}\delta_{kp})\\ 
&=& (\l_i \delta_{iq}\delta_{pl} +\l_p\delta_{pl}\delta_{iq} -\l_i\delta_{il}\delta_{pq}-\l_p \delta_{pq}\delta_{il} )\\ &&+  (\l_i \delta_{ip}\delta_{ql} +\l_q\delta_{ql}\delta_{ip} -\l_i\delta_{il}\delta_{pq}-\l_q\delta_{pq}\delta_{il} )\\
&=& (\l_p+\l_q) (\delta_{iq}\delta_{pl}+\delta_{ip}\delta_{ql}) \\
 &=& (\l_p+\l_q) (e_p \odot e_q)_{il},
\end{eqnarray*}
where $j$ and $k$ are summed from $1$ to $n$. 
\end{proof}

Since the orthogonal complement of $\spann\{e_p \odot e_q : 1 \leq p \neq q \leq n\}$ in $S^2_0(V)$ is $$E:=\left\{\sum_{p=1}^n c_p e_p \odot e_p: \sum_{p=1}^n c_p=0 \right\},$$ 
the eigentensors associated with the remaining eigenvalues are in $E$.
This reduces the problem of finding the remaining eigenvalues to solving some algebraic equations. 

\begin{lemma}\label{lemma 3.2}
Let $A\in S^2(V)$ and $\{e_i\}_{i=1}^n$ be an orthonormal basis of $V$ such that $A(e_i)=\l_i e_i$ for $1\leq i \leq n$. Suppose that $c_1, \cdots, c_n$, not all zeros, and $\l$ are real numbers satisfying
\begin{equation}\label{eq 3.1}
    \sum_{p=1}^n c_p=0, 
\end{equation}
and 
\begin{equation}\label{eq 3.2}
    2\l_pc_p -\frac{2}{n}\sum_{q=1}^n c_q\l_q =\l c_p \text{ for } 1\leq p \leq n.
\end{equation}
Then $\l$ is an eigenvalue of the curvature operator of the second kind of $A \owedge \id$ with eigentensor $\sum_{p=1}^n c_p e_p \odot e_p$.
\end{lemma}

\begin{proof}
For fixed $1\leq p\leq n$, we calculate that
\begin{eqnarray*}
 && \left((A \owedge \id ) (e_p \odot e_p)\right)_{il} \\
 &=&   (A \owedge \id )_{ijkl} (e_p \odot e_p)_{jk}  \\
&=&2  (A_{ik}\delta_{jl}+A_{jl}\delta_{ik}-A_{il}\delta_{jk}-A_{jk}\delta_{il}) \delta_{jp}\delta_{kp}  \\
&=& 2 (\l_i \delta_{ik}\delta_{jl}+\l_j \delta_{jl}\delta_{ik}-\l_i \delta_{il}\delta_{jk}-\l_j\delta_{jk}\delta_{il})\delta_{jp}\delta_{kp} \\    
&=& 2(\l_i \delta_{ip}\delta_{pl} +\l_p \delta_{pl}\delta_{ip}-\l_i \delta_{il} -\l_p \delta_{il}) \\
&=& 2\l_p (e_p \odot e_p)_{il} -2\l_p\delta_{il}-2\l_i \delta_{il},
\end{eqnarray*}
where $j$ and $k$ are summed from $1$ to $n$.
Namely, we have proven that
\begin{equation*}
    (A \owedge \id ) (e_p \odot e_p) =2\l_p e_p \odot e_p -2 \l_p g -\sum_{i=1}^n \l_i e_i \odot e_i,
\end{equation*}
for $1\leq p \leq n$. 
Next, we compute that
\begin{eqnarray*}
 && (A \owedge \id)\left(\sum_{p=1}^n c_p e_p\odot e_p\right) \\
 &=& 2 \sum_{p=1}^n \l_p c_p e_p\odot e_p -2\left(\sum_{p=1}^n  c_p \l_p  \right) g -\sum_{p=1}^n \left( c_p \sum_{i=1}^n \l_i e_i \odot e_i\right) \\
 &=& 2\sum_{p=1}^n \l_pc_p e_p\odot e_p -2\left(\sum_{p=1}^n c_p\l_p\right) g, 
\end{eqnarray*}
where we have used \eqref{eq 3.1} in the last step. We further calculate, with $\pi: S^2(V) \to S^2_0(V)$ being the projection map, that
\begin{eqnarray*}
&& \left( \pi \circ (A \owedge \id) \right)\left(\sum_{p=1}^n c_p e_p\odot e_p\right) \\ 
&=& \sum_{p=1}^n \left(2\l_pc_p  -\frac{2}{n} \sum_{q=1}^n c_q\l_q \right) e_p\odot e_p, \\
 &=& \l \sum_{p=1}^n c_p e_p\odot e_p,
\end{eqnarray*}
where we have used \eqref{eq 3.2} in the last step. This finishes the proof. 
\end{proof}

We are ready to prove Theorem \ref{thm diagonilization}. 
\begin{proof}[Proof of Theorem \ref{thm diagonilization}]
Let $\mu_1< \mu_2 < \cdots < \mu_k$ be the distinct eigenvalues of $A$ with corresponding multiplicities $n_1, \cdots, n_k$. 
Lemma \ref{lemma 3.1} implies that $\mu_i+\mu_j$ is an eigenvalue of the curvature operator of the second kind of $A\owedge \id$ with multiplicity $n_in_j$. 

By Lemma \ref{lemma 3.2}, we need to solve \eqref{eq 3.1} and \eqref{eq 3.2} to find the remaining eigenvalues. 
We first observe that for each $n_i >1$, $\l=2\mu_i$ is an eigenvalue of the curvature operator of the second kind of $A\owedge{\id}$ with multiplicity $n_i-1$ and the associated eigenspace is given by
\begin{eqnarray*}
    \spn \left\{\sum_{p=n_{i-1}+1}^{n_{i-1}+n_i} a_p e_p\odot e_p : \sum_{p=n_{i-1}+1}^{n_{i-1}+n_i} a_p =0 \right\}.
\end{eqnarray*}
Here and in the rest of the proof, we use the convention $n_0=0$ for simplicity of notations.
In this way, we can find $(n_1-1)+\cdots + (n_k-1)=n-k$ eigenvalues. 

To find the remaining $k-1$ eigenvalues, we observe that if $\l$ is a solution to the equation 
\begin{equation}\label{eq 3.3}
    \sum_{i=1}^k \frac{n_i \mu_i}{2\mu_i-\l}=\frac{n}{2},
\end{equation}
then $\l$ and $c_1, \cdots, c_n$ defined by 
$$c_{n_{i-1}+1}=\cdots =c_{n_{i-1}+n_i}=\frac{1}{2\mu_i-\l} \text{ for } 1\leq i \leq k$$ 
satisfy \eqref{eq 3.2}.
Note that \eqref{eq 3.1} is also satisfied if either $\l \neq 0$ or $\l=0$ and $\sum_{i=1}^k \frac{n_i}{\mu_i}=0$. By Lemma \ref{lemma 3.2}, nonzero solutions to \eqref{eq 3.3} are eigenvalues of the curvature operator of the second kind of $A \owedge \id$. Also, $\l=0$ is an eigenvalue if $\sum_{i=1}^k \frac{n_i}{\mu_i}=0$. 

If $\mu_p=0$ for some $1\leq p \leq k$, we claim that \eqref{eq 3.3} has exactly one  solution in $(\mu_i, \mu_{i+1})$ for each $1\leq i \leq k-1$. To see this, we consider the function $f$ defined by
\begin{equation*}
    f(\l)=  \sum_{i=1}^k \frac{n_i \mu_i}{2\mu_i-\l}-\frac{n}{2}. 
\end{equation*}
Clearly, $\lim_{\l \to 0}f(\l) =-\frac{n_p}{2} <0$. 
Moreover, if $\mu_i >0$, then
\begin{equation*}
    \lim_{\l \to 2\mu_i^-} f(\l)=\infty \text{ and } \lim_{\l \to 2\mu_i^+} f(\l) = -\infty,
\end{equation*}
and if $\mu_i <0$, then
\begin{equation*}
    \lim_{\l \to 2\mu_i^-} f(\l)=-\infty \text{ and } \lim_{\l \to 2\mu_i^+} f(\l) = \infty.
\end{equation*}
By the intermediate value theorem, the continuity of $f$ on $\R \setminus \{2\mu_1, \cdots, 2\mu_k\}$, and the asymptotics of $f$ near the $2\mu_i$'s, one sees that for each $1\leq i\leq k-1$, $f(\l)=0$ must have at least one solution in the interval $(\mu_i, \mu_{i+1})$ . 
Since each solution of $f(\l)=0$ is also a solution of the degree $k-1$ polynomial 
\begin{equation*}
    f(\l) \prod_{1\leq i \leq k, i\neq p} (2\mu_i -\l),
\end{equation*}
one concludes that $f(\l)=0$ has exactly one solution on $(\mu_i, \mu_{i+1})$ for each $1\leq i\leq k-1$. Therefore, \eqref{eq 3.3} has $k-1$ distinct nonzero solutions. 

Next, we consider the case $\mu_i\neq 0$ for all $1\leq i\leq k$. 
In this case, $f$ has the same asymptotics at the $2\mu_i$'s as before, but $f(0)=0$. On an interval of the form $(\mu_i, \mu_{i+1})$ not containing $0$, the intermediate value theorem implies that $f(\l)=0$ has at least one solution in it. If $0 \in (\mu_i, \mu_{i+1})$, then $f(\l)=0$ has at least one solution in $(\mu_i, 0)$ if $f'(0) > 0$ and at least one solution in $(0, \mu_{i+1})$ if $f'(0)< 0$. 
When $f'(0)=\frac{1}{4}\sum_{i=1}^k \frac{n_i}{\mu_i}=0$, zero is a solution of $f(\l)=0$ with multiplicity two. In this case, $0$ is an eigenvalue of the curvature operator of the second of $A\owedge \id$, as $\l=0$ and $c_1, \cdots, c_n$ defined by 
$$c_{n_{i-1}+1}=\cdots =c_{n_{i-1}+n_i}=\frac{1}{2\mu_i} \text{ for } 1\leq i \leq k,$$ 
satisfies both \eqref{eq 3.1} and \eqref{eq 3.2}. 
Noticing that each solution of $f(\l)=0$ is also a solution of the degree $k$ polynomial 
\begin{equation*}
    f(\l) \prod_{i=1}^k (2\mu_i -\l),
\end{equation*}
we conclude that if $\sum_{i=1}^k \frac{n_i}{\mu_i}\neq 0$, then $f(\l)=0$ has exactly one solution in $(\mu_i,\mu_{i+1})$ for each $1\leq i \leq k-1$, and if $\sum_{i=1}^k \frac{n_i}{\mu_i}=0$, then $f(\l)=0$ has exactly one solution in each interval of the form $(\mu_i,\mu_{i+1})$ not containing zero, and $0$ is a solution of $f(\l)=0$ of multiplicity two. 

Overall, the remaining $k-1$ eigenvalues of the curvature operator of the second kind of $A\owedge \id$ are given by the $k-1$ nonzero solutions of \eqref{eq 3.3} if $\sum_{i=1}^k \frac{n_i}{\mu_i}\neq 0$, and by the $k-2$ nonzero solutions of \eqref{eq 3.3} together with $0$ if $\sum_{i=1}^k \frac{n_i}{\mu_i}= 0$. In both cases, there is exactly one eigenvalue in $(\mu_i, \mu_{i+1})$ for each $1\leq i \leq k$, and all eigenvalues lie in the interval $(\mu_1, \mu_k)$.

The proof is complete.
\end{proof}

It is clear from the proof of Theorem \ref{thm diagonilization} that
\begin{corollary}
If the eigenvalues of $A\in S^2(V)$ lie in the interval $[a,b]$, then the eigenvalues of the curvature operator of the second kind of $A\owedge \id $ lie in $[2a,2b]$. 
\end{corollary}

\section{Proof of Theorem \ref{thm eigenvalue 3D}}

We apply Theorem \ref{thm diagonilization} to dimension three and prove Theorem \ref{thm eigenvalue 3D}. 
\begin{proof}[Proof of Theorem \ref{thm eigenvalue 3D}]

Note that $R\in S^2_B(\wedge^2 \R^3)$ can be written as
\begin{equation*}
    R=A\owedge g,
\end{equation*}
where $A:=\Ric-\frac{S}{4}g$ is the Schouten tensor. 
If $a\leq b \leq c$ are the eigenvalues of $\hat{R}$, then the eigenvalues of $\Ric$ are $a+b\leq a+c \leq b+c$ and the scalar curvature is $S=2(a+b+c)$. Thus, the eigenvalues of $A$ are given by 
$$\frac{1}{2}(a+b-c) \leq \frac{1}{2}(a+c-b)\leq \frac{1}{2}(b+c-a).$$

Let's first deal with the case $a<b<c$. By Theorem \ref{thm diagonilization}, the eigenvalues of $\mathring{R}$ are $a< b < c$, and the two solutions of the equation
$$\frac{a+b-c}{a+b-c -\l} + \frac{a+c-b}{a+c-b -\l } + \frac{b+c-a}{b+c-a-\l}=3,$$
which are given by
$$\l_{\pm}=\frac{a+b+c}{3}\pm \frac{\sqrt{2}}{3}\sqrt{3(a^2+b^2+c^2)-(a+b+c)^2}.$$ 
Next, one verifies that the expressions of eigenvalues of $\mathring{R}$ in Theorem \ref{thm eigenvalue 3D} remain valid for the cases $a=b<c$, $a<b=c$, and $a=b=c$.  
Finally, the inequalities $\l_{-}\leq a$ and $\l_{+}\geq c$ follow from simple algebraic manipulations.

\end{proof}

Next, we prove Proposition \ref{prop 3D curvature}. 
\begin{proof}
Parts (1), (3), and (5) have been proved in \cite{Li21, Li22JGA}, so we only prove parts (2) and (4) here. 

Let $R\in S^2_B(\wedge^2 \R^3)$ and let $a\leq b \leq c$ be the eigenvalues of $\hat{R}$. 
Then $R$ has nonnegative sectional curvature (or $\hat{R}\geq 0$) if and only if $a\geq 0$, and $R$ has nonnegative Ricci curvature (or $\hat{R}$ is two-nonnegative) if and only if $a+b\geq 0$. By Theorem \ref{thm eigenvalue 3D}, the eigenvalues of $\mathring{R}$ are given by
$$\l_- \leq a \leq b \leq c \leq \l_+,$$ 
where $\l_{\pm}$ are defined in \eqref{def lambda pm}. 

Note that $\mathring{R}$ is $3\frac{1}{3}$-nonnegative if and only if $\l_{-} +a +b +\frac{c}{3}\geq 0$, which is, after some algebraic manipulations, equivalent to $2a+2b+c \geq 0$ and
$$12a^2+12b^2+36ab+20ac+20bc \geq 0.$$ 
Both inequalities hold if $a \geq 0$. So, we have proved part (2).

To prove part (4), we observe that the inequality $ \l_{-} +a +b +c\geq 0$ is equivalent to $a+b+c\geq 0$ and 
\begin{equation*}
    (a+b)^2+(c^2 +ab)+3(a+b)c \geq 0,
\end{equation*}
which holds provided that $a+b\geq 0$.
Results with ``nonnegative" replaced by ``positive" follows similarly. 

To prove the statement for nonpositivity in (2), we note that $R$ has nonpositive sectional curvature $\iff$ $-R$ has nonnegative sectional curvature $\implies$ $-\mathring{R}$ is $3\frac{1}{3}$-nonnegative $\iff$ $\mathring{R}$ is $3\frac{1}{3}$ is nonpositive. Other results concerning negativity or nonpositivity can be deduced similarly.

\end{proof}


Below we construct some examples to show the sharpness of Proposition \ref{prop 3D curvature}. 
In the following, $a\leq b \leq c$ are the eigenvalues of $\hat{R}$ and $\epsilon$ is a small positive number. 
\begin{example}
Let $a=-\epsilon$ and $b=c=1$. Then $\l_{-}=-\epsilon$ and $\l_+=\frac{4+\epsilon}{3}$. This provides an example of $R\in S^2_B(\wedge^2 \R^3)$ such that $\mathring{R}$ is $(2+2\epsilon)$-nonnegative but $R$ does not have nonnegative sectional curvature. 
\end{example}


\begin{example}
Let $a=b=0$ and $c=1 $. Then $\l_{-}=-\frac{1}{3}$ and $\l_+=1$. This provides an example of $R\in S^2_B(\wedge^2 \R^3)$ with nonnegative sectional curvature but $\mathring{R}$ is not $\a$-nonnegative for any $\a < 3\frac{1}{3}$. This example is the curvature tensor of $\mathbb{S}^2 \times \R$ with the product metric. 
\end{example}

\begin{example}
Let $a=-\epsilon$, $b=0$, and $c=1+\epsilon $. Then $\l_{\pm}=\frac{1}{3}\pm \frac{2}{3}\sqrt{1+3\epsilon^2+3\epsilon}$. This provides an example of $R\in S^2_B(\wedge^2 \R^3)$ such that $\mathring{R}$ is $\left(3\frac{1}{3}+\delta(\epsilon)\right)$-nonnegative but $R$ does not have nonnegative Ricci curvature, where $\delta(\epsilon)=\frac{2}{3}\frac{\sqrt{1+3\epsilon^2+3\epsilon}-(1-\epsilon)}{1+\epsilon}$.
\end{example}



\begin{example}
Let $a=-1$ and $b=c=1$. Then $\l_{-}=-1$ and $\l_+=\frac{5}{3}$. This provides an example of $R\in S^2_B(\wedge^2 \R^3)$ such that $\Ric \geq 0$ but $\mathring{R}$ is not $\a$-nonnegative for any $\a < 4$. 
\end{example}

In order to prove part (1) of Theorem \ref{thm 3-manifold classification}, we need a proposition.
\begin{proposition}\label{prop 3.3}
Let $R\in S^2_B(\wedge^2 \R^3)$ be an algebraic curvature operator. If $\mathring{R}$ is $(3+\delta)$-nonnegative for some $\delta \in [0,\frac{1}{3}]$, then 
\begin{equation*}
    \Ric \geq \frac{1-3\delta}{3(2-\delta)} S.
\end{equation*}
\end{proposition}
\begin{proof}
As in \cite[Section 3]{Li21}, we choose an orthonormal basis of $S^2_0(\R^3)$ as follows: 
\begin{eqnarray*}
\vp_1 &=& \frac{1}{\sqrt{2}} e_1 \odot e_2, \\
\vp_2 &=& \frac{1}{\sqrt{2}} e_1 \odot e_3, \\
\vp_3 &=& \frac{1}{\sqrt{2}} e_2 \odot e_3, \\
\vp_4 &=& \frac{1}{2\sqrt{2}} \left( e_1 \odot e_1- e_2 \odot e_2 \right) , \\
\vp_5 &=& \frac{1}{2\sqrt{6}} \left( e_1 \odot e_1+e_2 \odot e_2 -2e_3 \odot e_3 \right), \\
\end{eqnarray*}
where $\{e_1,e_2,e_3\}$ is an orthonormal basis of $\R^3$. 
By \cite[Lemma 3.1]{Li21}, we have that $\mathring{R}(\vp_1,\vp_1) = \mathring{R}(\vp_4,\vp_4) =R_{1212}$, $\mathring{R}(\vp_2,\vp_2)=R_{1313}$, $\mathring{R}(\vp_3,\vp_3)=R_{2323}$, and $\mathring{R}(\vp_5,\vp_5)=\frac{2}{3}(R_{1313}+R_{2323})-\frac{1}{3}R_{1212}$.

Since $\mathring{R}$ is $(3\frac{1}{3}+\delta)$-nonnegative,we get
\begin{eqnarray*}
0& \leq & \mathring{R}(\vp_2,\vp_2)+\mathring{R}(\vp_3,\vp_3)+\mathring{R}(\vp_5,\vp_5) +\delta \mathring{R}(\vp_1,\vp_1)  \\
&=& \frac 5 3 (R_{1313}+R_{2323}) -\left( \frac{1}{3}-\delta \right) R_{1212}  \\
&=& \frac 5 3 R_{33}-\left( \frac{1}{3}-\delta \right) \left( \frac{S}{2}-R_{33}\right)\\
&=& (2 -\delta) R_{33} -\left( \frac{1}{3}-\delta \right)\frac{S}{2}.
\end{eqnarray*}
Thus, we have $ \Ric \geq \frac{1-3\delta}{3(2-\delta)}S$.
\end{proof}


\begin{proof}[Proof of Theorem \ref{thm 3-manifold classification}]
(1). If $M$ has $\a$-nonnegative $\mathring{R}$ for some $\a \in [1,3\frac{1}{3})$, then it has $(3+\delta)$-nonnegative $\mathring{R}$ for $\delta=0$ if $\a \in [1,3]$ and $\delta =3\frac{1}{3}-\a >0$ if $\a \in [3,3\frac{1}{3})$. By Proposition \ref{prop 3.3}, we have $\Ric \geq \frac{1-3\delta}{3(2-\delta)} S$. 

It was conjectured by Hamilton and proved in \cite{CZ00}, \cite{Lott19}, \cite{DSS22}, and \cite{LT22} that a three-manifold with $\Ric \geq \epsilon S$ for some $\epsilon >0$ is either flat or compact. Therefore, $M$ is compact and has positive Ricci curvature unless it is flat. The desired conclusion then follows from Hamilton's classification of three-manifolds with positive Ricci curvature \cite{Hamilton82}.

Since $3\frac{1}{3}$-nonnegativity of $\mathring{R}$ implies nonnegative Ricci curvature, parts (2) and (3) follow from the classification of three-manifolds with nonnegative Ricci obtained by Hamilton \cite{Hamilton86} in the closed case and by Liu \cite{Liu13} in the complete noncompact case, respectively. 

For some $\a \in (3\frac{1}{3}, 5]$, $\a$-nonnegativity of $\mathring{R}$ implies nonnegative scalar curvature. Therefore, part (4) follows from the classification of compact three-manifolds with nonnegative scalar curvature, which is a consequence of the work of Perelman \cite{Perelman1, Perelman2, Perelman3}.
\end{proof}
\section{preserving $\a$-nonnegativity of $\mathring{R}$}

In this section, we show that $\a$-nonnegative/$\a$-positive curvature operator of the second kind is preserved by compact Ricci flows in dimension three for any $\a \in [1,5]$.
In view of Hamilton's ODE-PDE maximum principle (see \cite{Hamilton86}), it suffices to show that Hamilton's ODE
\begin{equation}\label{Hamilton ODE}
    \frac{d R}{dt} = R^2 + R^{\sharp}.
\end{equation}
preserves $\a$-nonnegativity/$\a$-positivity of $\mathring{R}$.

Given $R \in S^2(\wedge^2 \R^3)$, let $a\leq b\leq c$ denote the eigenvalues of $\hat{R}$.
By Theorem \ref{thm eigenvalue 3D}, the eigenvalues of $\mathring{R}$ are given by 
$$\l_- \leq a \leq b \leq c \leq \l_+,$$ 
where $\l_{\pm}$ are defined in \eqref{def lambda pm}. 
We define
\begin{equation}\label{def f_a}
    f_{\a}(R) :=
    \begin{cases} 
    \l_- +(\a-1)a, & \text{ if } \a \in [1,2), \\
    \l_-+a +(\a-2)b, & \text{ if } \a \in [2,3), \\
   \l_-+a+b+(\a-3)c, & \text{ if } \a \in [3,4), \\
    \l_-+a+b+c + (\a-4)\l_+, & \text{ if } \a \in  [4,5).
    \end{cases}
\end{equation}
Clearly, we have
\begin{proposition}\label{prop equiv a positive}
$\mathring{R}$ is $\a$-nonnegative (respectively, $\a$-positive) if and only if $f_\a(R) \geq 0$ (respectively, $f_\a(R) >0$). 
\end{proposition}

Let $R(t)$, $t\in [0,T)$, be a solution to \eqref{Hamilton ODE}. 
By \cite{Hamilton82}, the eigenvalues of $\hat{R}$ evolve by the ODE system
\begin{align}\label{eq evolution eigenvalues first kind} \nonumber
    \frac{da}{dt} &= a^2+bc, \\ 
    \frac{db}{dt} &= b^2+ac, \\ \nonumber
    \frac{dc}{dt} &= c^2+ac,\\ \nonumber
\end{align}
and the scalar curvature $S(t)$ satisfies 
\begin{equation}\label{eq S evolution}
    \frac{d}{dt} S =2(a^2+b^2+c^2+ab+ac+bc)= |\Ric|^2.
\end{equation}
If $\mathring{R}(0)$ is $\a$-nonnegative for some $\a \in [1,5]$, then $S(0) \geq 0$. By \eqref{eq S evolution}, $S(t)\geq 0$ for each $t\in [0,T)$. Moreover, we have either $S(t) \equiv 0$ or $S(t)>0$ for each $t\in [0,T)$. Since there is nothing to prove if $S(t)\equiv 0$, we may assume $S(t)>0$ in the following discussions.


We first prove that
\begin{proposition}\label{prop nondecreasing}
Let $R(t)$, $t \in [0,T)$, be a solution to \eqref{Hamilton ODE} with $S(t)>0$. 
Then the quantities 
\begin{equation*}
    \frac{a}{S}, \frac{a+b}{S}, -\frac{c}{S}, \frac{\l_-}{S}, -\frac{\l_+}{S}, -\frac{|\Ric|^2}{S^2},
\end{equation*}
are monotone non-decreasing in $t$.
\end{proposition}
\begin{proof}
Using \eqref{eq evolution eigenvalues first kind} and \eqref{eq S evolution}, we derive that
\begin{align*}
\frac{d}{dt} \left(\frac{a}{S} \right)&= \frac{2}{S^2}(b^2(c-a)+c^2(b-a)) \geq 0, \\
 \frac{d}{dt} \left(\frac{a+b}{S}\right) &= \frac{2}{S^2}(a^2(c-b)+b^2(c-a)) \geq 0, \\
 \frac{d}{dt} \left(\frac{c}{S}\right) &= \frac{2}{S^2}(a^2(b-c)+b^2(a-c)) \leq 0.
\end{align*}
This implies the desired monotonicity of $\frac{a}{S}$, $\frac{a+b}{S}$, and $\frac{c}{S}$.

Using $|\Ric|^2=2(a^2+b^2+c^2+ab+ac+bc)$ and \eqref{eq evolution eigenvalues first kind}, we compute that
\begin{equation*}
    \frac{d}{dt} |\Ric|^2 =4(a^3+b^3+c^3+3abc+ab^2+a^2b+ac^2+a^2c+bc^2+b^2c), 
\end{equation*}
and
\begin{equation*}
    \frac{d}{dt} \frac{|\Ric|^2}{S^2} 
    =-\frac{4}{S^3}\left((a^2(b-c)^2+b^2(a-c)^2+c^2(a-b)^2)   \right) \leq 0,
\end{equation*}
yielding the monotonicity of $\frac{|\Ric|^2}{S^2}$. 

Finally, the monotonicity of $\frac{\l_{\pm}}{S}$ follows from the identity
\begin{equation*}
    \frac{\l_{\pm}}{S} = \frac{1}{6} \pm \frac{\sqrt{2}}{3}\sqrt{3\frac{|\Ric|^2}{S^2} -1}
\end{equation*}
and the monotonicity of $\frac{|\Ric|^2}{S^2}$. 
\end{proof}

\begin{proposition}\label{prop f_a lower bound}
Let $R(t)$, $t \in [0,T)$, be a solution to \eqref{Hamilton ODE} with $S(t)>0$. Let $f_\a(R)$be the function defined in \eqref{def f_a}.
If $f_\a(0) \geq c S(0)$ for some $c\in \R$, then $f_\a(t) \geq c S(t)$ for all $t\in [0,T)$. 
\end{proposition}
\begin{proof}
Note that $\frac{f_\a(R)}{S}$ can be written as 
%
\begin{equation*}
    \frac{f_\a(R)}{S}=
    \begin{cases} 
    \frac{\l_-}{S} +(\a-1)\frac{a}{S}, & \text{ if } \a \in [1,2), \\
    \frac{\l_-}{S}+(3-\a)\frac{a}{S} +(\a-2)\frac{a+b}{S}, & \text{ if } \a \in [2,3), \\
    \frac{\l_-}{S}+\frac{1}{2}+(\a-4)\frac{c}{S}, & \text{ if } \a \in [3,4), \\
    \frac{1}{2} + \frac{\a-4}{3} +(5-\a)\frac{\l_+}{S}, & \text{ if } \a \in  [4,5].
    \end{cases}
\end{equation*}
In other words, $\frac{f_\a(R)}{S}$ is the sum of functions that are non-decreasing along \eqref{Hamilton ODE}.
Therefore, $\frac{f_\a(R)}{S}$ is non-decreasing in $t$. 
\end{proof}

Combining Proposition \ref{prop f_a lower bound} and Proposition \ref{prop equiv a positive}, we get that

\begin{proposition}\label{prop RF preserves second kind}
The ODE \eqref{Hamilton ODE} preserves $\a$-nonnegativity/$\a$-positivity of $\mathring{R}$ for any $\a \in [1,5]$. 
\end{proposition}

Next, we prove Theorem \ref{thm RF preserves}.
\begin{proof}[Proof of Theorem \ref{thm RF preserves}]
Take $\a \in [1,5],$ $\varepsilon \geq 0$ and set 
$$K_{\a,\varepsilon}:=\{ R\in S^2_B(\wedge^2 \R^3): f_\a(R) \geq \varepsilon S\geq 0 \}.$$ Clearly, $K_{\a,\varepsilon}$ is closed, $O(3)$-invariant, and invariant under parallel transport. 
To see that $K_{\alpha,\epsilon}$ is convex, we observe that 
\begin{equation*}
    f_\a(R)= \min\{\mathring{R}(\vp_1,\vp_1)+\cdots +  \mathring{R}(\vp_{\lfloor \a \rfloor},\vp_{\lfloor \a \rfloor}) + (\a-\lfloor \a \rfloor) \mathring{R}(\vp_{\lfloor \a \rfloor+1}, \vp_{\lfloor \a \rfloor+1})   \},
\end{equation*}
where the minimum is taken over all orthonormal bases $\{\vp_i\}_{i=1}^5$ for $S^2_0(\R^3)$. Since the composition of the pointwise minimum with a linear map of functionals is concave, we conclude that $f_\a(R)$ is concave. Together with the fact that scalar curvature is a linear function on $S^2_B(\wedge^2 \R^3)$, we conclude that 
$$K_{\a,\varepsilon}=(f_\alpha(R) -\varepsilon S)^{-1}([0,\infty))\cap S^{-1}([0,\infty))$$ 
is a convex set. 
By combining Proposition \ref{prop RF preserves second kind} with Hamilton's ODE-PDE  maximum principle, we conclude that $K_{\a,\varepsilon}$ is preserved by the Ricci flow in dimension three.
\end{proof}

Hamilton \cite{Hamilton82} proved that (compact) Ricci flows in dimension three preserve positivity, two-positivity, and three-positivity of $\hat{R}$, which correspond to positive sectional curvature, positive Ricci curvature, and positive scalar curvature, respectively. 
Here, we note that a similar argument as in the proof of Theorem \ref{thm RF preserves} shows that compact three-dimensional Ricci flows preserve $\a$-nonnegativity of $\hat{R}$ for any $\a \in [1,3]$.
\begin{proposition}\label{prop RF preserves first kind}
Let $(M^3,g(t))$, $t\in [0,T)$, be a compact three-dimensional Ricci flow. If $g(0)$ has $\a$-nonnegative (respectively, $\a$-positive) curvature operator for some $\a \in [1,3]$, then $g(t)$ has $\a$-nonnegative (respectively, $\a$-positive) curvature operator for each $t\in (0,T)$. 
\end{proposition}
\begin{proof}
Denote by $a\leq b \leq c$ the eigenvalues of $\hat{R}$. Clearly, $\hat{R}$ is $\a$-nonnegative (respectively, $\a$-positive) if and only if $h_\a(R) \geq 0$ (respectively, $>0$), where
\begin{equation*}
    h_\a(R):=\begin{cases}
    a +(\a-1) b, & \text{ if } \a \in [1,2), \\
    a+b+(\a-2)c,  & \text{ if } \a \in [2,3]
    \end{cases}
\end{equation*}
Since $S(t_0)=0$ forces $S(t) \equiv 0$, we may assume $S(t)>0$ for $t\in [0,T)$. 
Note that 
\begin{equation*}
    \frac{h_\a(R)}{S}=\begin{cases}
    (2-\a)\frac{a}{S} +(\a-1) \frac{a+b}{S}  & \text{ if } \a \in [1,2), \\
    \frac{1}{2}+(\a-3)\frac{c}{S}, & \text{ if } \a \in [2,3].
    \end{cases}
\end{equation*}
The monotonicity of $\frac{a}{S}$, $\frac{a+b}{S}$, and $\frac{c}{S}$ obtained in Proposition \ref{prop nondecreasing} implies that the function $\frac{h_\a(R)}{S}$ is monotone non-decreasing under Hamilton's ODE \eqref{Hamilton ODE}. 
The rest of the proof uses Hamilton's ODE-PDE maximum principle 
as in the proof of Theorem \ref{thm RF preserves}. 
\end{proof}

\section*{Acknowledgments}
The authors would like to thank Professor Xiaodong Cao for many helpful discussions. We also appreciate the anonymous referee for his/her comments which improved the paper's exposition.

\section*{Conflict of Interest}
On behalf of all authors, the corresponding author states that there is no conflict of interest.

\section*{Data Availability Statement}
Data sharing is not applicable to this article as no datasets were generated or analyzed during the current study.

\bibliographystyle{alpha}
\bibliography{ref}

\end{document}